\documentclass[11pt]{amsart}
\usepackage[margin=1in]{geometry}

\usepackage{amssymb}
\usepackage{amsthm}
\usepackage{amsmath}
\usepackage{mathrsfs}
\usepackage{amsbsy}
\usepackage[all]{xy}
\usepackage{bm}
\usepackage{hyperref}
\usepackage{tikz}
\usepackage{array}
\usepackage{float}
\usepackage{enumerate}
\usepackage{xcolor}
\usepackage{hhline}
\usepackage{mathtools}
\allowdisplaybreaks
\usepackage[noadjust]{cite}

\usepackage{caption}
\usepackage{subcaption}
\usepackage{diagbox}
\usepackage{tikz}
\usepackage{bbm}
\usepackage{booktabs}

\usepackage[noabbrev,capitalise]{cleveref}

\newtheorem{thm}{Theorem}[section]
\newtheorem{lem}[thm]{Lemma}
\newtheorem{cor}[thm]{Corollary}
\newtheorem{prop}[thm]{Proposition}

\newtheorem{rem}[thm]{Remark}
\newtheorem*{thm*}{Theorem \ref{th:main3}}
\newtheorem*{thm*2}{Theorem \ref{th:main1}}

\theoremstyle{definition}
\newtheorem{defn}[thm]{Definition}

\begin{document}
\title[]{Anticoncentration of Random Sums in $\mathbb{Z}_p$}
\author[]{Simone Costa}
\address[Simone Costa]{DICATAM, Universit\`a degli Studi di Brescia, Via Branze~43, I~25123 Brescia, Italy}
\email{simone.costa@unibs.it}

\maketitle
\begin{abstract}
In this paper we investigate the probability distribution of the sum $Y$ of $\ell$ independent identically distributed random variables taking values in $\mathbb{Z}_p$. Our main focus is the regime of small values of $\ell$, which is less explored compared to the asymptotic case $\ell \to \infty$.

Starting with the case $\ell=3$, we prove that if the distributions of the $Y_i$ are uniformly bounded by $\lambda < 1$ and $p > 2/\lambda$, then there exists a constant $C_{3,\lambda} < 1$ such that
\[
\max_{x \in \mathbb{Z}_p} \mathbb{P}[Y = x] \leq C_{3,\lambda}\lambda.
\]
Moreover, when the distributions are uniformly separated from $1$, the constant $C_{3,\lambda}$ can be made explicit. By iterating this argument, we obtain effective anticoncentration bounds for larger values of $\ell$, yielding nontrivial estimates already in small and moderate regimes where asymptotic results do not apply.
\end{abstract}

\section{Introduction}

The Littlewood–Offord problem is a classical question in probabilistic combinatorics concerning the anticoncentration of sums of independent random variables. It was introduced by Littlewood and Offord in the 1940s (see \cite{LO}) in order to study the distribution of sums of random variables with restricted supports. In its original formulation, given a list of (not necessarily distinct) integers, the problem asks for upper bounds on the probability of obtaining a prescribed value by summing elements from the list. Equivalently, given integers $(v_1, v_2, \dots, v_{\ell})$, one studies the distribution of $\sum_{i=1}^{\ell} Y_i$, where each $Y_i$ is uniformly distributed on $\{0,v_i\}$ (or, in an equivalent formulation, on $\{-v_i,v_i\}$). In \cite{LO}, Littlewood and Offord proved an upper bound of order $O(\frac{\log n}{\sqrt{n}})$, which was later improved by Erd\H{o}s to $\frac{1}{\sqrt{n}}(1+o(1))$ in \cite{E}.

The problem and its variants arise in the analysis of random walks, random matrices, and other combinatorial structures (see for instance \cite{B,TV2}). Significant progress has been made under various structural assumptions on the distributions of the variables (see \cite{JK,LR}), as well as in settings involving finite groups. In particular, Vaughan and Wooley \cite{VW} considered the case where the variables $Y_i$ are uniform on $\{0,v_i\}$ in cyclic groups, and further bounds were obtained by Griggs \cite{G}, Bibak \cite{B}, and Juskevicius and Semetulskis \cite{JS}. In the latter work, the authors also investigated the case where the distributions of the $Y_i$ are uniformly bounded by $1/2$.

The present work is inspired by these developments. We consider independent random variables $Y_1,\dots,Y_{\ell}$ taking values in a cyclic group $\mathbb{Z}_k$, whose distributions are pointwise bounded by $\lambda$.

If the distributions are uniform on a subset $A\subseteq \mathbb{Z}_k$ with $|A|=n$, the problem is closely related to estimating the probability that a subset $X\subseteq A$ of fixed cardinality $\ell$ has sum equal to $x$. This can be viewed as a variation of the classical Littlewood–Offord problem in which the size of the subset is fixed and the elements $v_1,\dots,v_n$ are distinct. The case of distinct integers was studied by Erd\H{o}s and later by Halász in \cite{H} (see also \cite{TV1}), where bounds of order $O(\frac{1}{n\sqrt{n}})$ were obtained; very recent results have been obtained for this problem also in $\mathbb{Z}_p$ by Pham and Sauermann (see \cite{PM}).

These questions are also motivated by applications to the set sequenceability problem (see \cite{HOS19,CDOR} and \cite{PM}), where one needs anticoncentration estimates in finite cyclic groups in regimes where both $p$ and $\ell$ are sufficiently large. In this context, it was brought to our attention that Lev established very strong asymptotic results in \cite{L1,L2}. In particular, in \cite{L2} he proved that in $\mathbb{Z}_p$ one has
$$
\mathbb{P}[Y=x]\leq \frac{1}{p}+\frac{1}{n}\sqrt{\frac{8}{\pi\ell}}\left(1+2\ell^{-1/2}+(3/4)^{\ell/2+3}\ell^{3/2}\right).
$$
Thus, for $p$ and $\ell$ sufficiently large, one recovers the same qualitative behaviour as in the integer setting.

However, a direct inspection shows that for $\ell<24$ the above estimate does not improve upon the trivial bound. The aim of this paper is to address precisely this non-asymptotic regime. We will moreover treat distributions that are not necessarily uniform.

The paper is organized as follows. In Section~2 we revisit the integer case. We present a direct and self-contained proof of a bound that appears in \cite{JK} in a more general framework, and we make the constants explicit, improving in particular the case $\ell=3$. Denoting by $n$ the cardinality of the support, we show that there exists an absolute constant $D$ such that
$$
\max_{x\in \mathbb{Z}} \mathbb{P}[Y=x]\leq\frac{D}{n\sqrt{\ell-1}}.
$$
In particular, for every $\epsilon>0$, if $\ell$ is sufficiently large with respect to $\epsilon$, then
$$
\max_{x\in \mathbb{Z}} \mathbb{P}[Y=x]\leq\frac{\epsilon}{n}.
$$
Moreover, in the case $\ell=3$, we obtain the explicit bound $\epsilon=\frac{3+1/n^2}{4}$, which is already non-trivial.

We then turn to cyclic groups. When the prime factors of $k$ are sufficiently large, a Freiman isomorphism of order $\ell$ allows us to transfer the integer bound to $\mathbb{Z}_k$. To avoid assumptions on the prime factors of $k$, one may alternatively use Lev’s results.

Our main contribution is contained in Section~3, where we obtain non-trivial anticoncentration bounds for every $\ell\geq 3$. More precisely, if $p>\frac{2}{\lambda}\left(\frac{\ell_0}{3}\right)^{\nu}$ (where $\ell_0$ is a power of three such that $\ell_0\leq \ell$), and if $Y=Y_1+\dots+Y_{\ell}$ with $Y_i$ i.i.d.\ and distributions uniformly bounded by $\lambda\leq 9/10$ (this restriction is only needed to provide an explicit value of $\nu$), then there exists an absolute constant $\nu>0$ such that
$$
\max_{x\in \mathbb{Z}_p} \mathbb{P}[Y=x]\leq\lambda\left(\frac{3}{\ell_0}\right)^{\nu}.
$$

In summary, while Lev’s asymptotic methods dominate in the large-$\ell$ regime,
our approach provides explicit and self-contained bounds for small and moderate
values of $\ell$, and applies also to distributions that are not necessarily uniform,
where asymptotic estimates are not yet effective.

\section{Known inequalities}
In this section, given a set $A=\{v_1,v_2,\dots,v_n\}$ of distinct integers, we consider the random variable $Y$ given by the sum of $\ell$ independent variables $Y_1,\dots, Y_{\ell}$ that are uniformly distributed on $A$. We establish an upper bound on $\mathbb{P}[Y=x]$ corresponding to a special case of Theorem~2.3 in \cite{JK}. Although this follows from their more general result, we provide a direct and self-contained proof, which we believe may be of independent interest.

First of all we restrict ourself to consider $A=\{-(n-1)/2,\dots,(n-1)/2\}$, indeed Theorem 2 of \cite{LR} state that:
\begin{thm}[Leader and Radcliffe]\label{LRthm} Denoted by $\tilde{Y}$ the random variable given by the sum of $\ell$ independent and uniformly distributed on $\{-(n-1)/2,\dots,(n-1)/2\}$ variables $\tilde{Y}_1,\dots, \tilde{Y}_{\ell}$ and considered $Y$ defined as above, we have that
$$\left(\max_{x\in \mathbb{Z}} \mathbb{P}[Y=x]\right)\leq\left(\max_{x\in \mathbb{Z}} \mathbb{P}[\tilde{Y}=x]\right)=\mathbb{P}[\tilde{Y}=M]$$
where $M=\begin{cases}0 \mbox{ if } \ell\equiv 0 \pmod{2} \mbox{ or }n\equiv 1 \pmod{2};\\
-\frac{1}{2} \mbox{ otherwise}.\end{cases}$
\end{thm}
Then our main tool to upper-bound $\mathbb{P}[\tilde{Y}=x]$ is the Berry-Esseen theorem which says that:
\begin{thm}[Berry-Esseen]\label{BE}
Let $Y_1,\dots, Y_{\ell'}$ be independent random variables with the same distribution and such that:
\begin{itemize}
\item[(a)] $\mathbb{E}(Y_j)=0$;
\item[(b)] $\mathbb{E}(Y_j^2)=\sigma^2>0$;
\item[(c)] $\mathbb{E}(|Y_j|^3)<\infty$.
\end{itemize}
Then, set $\rho=\frac{\mathbb{E}(|Y_j|^3)}{\sigma^3}$ we have

$$\sup_x \left| \mathbb{P}[\frac{1}{\sigma\sqrt{\ell'}}\sum_{j=1}^{\ell'} Y_j\leq x]-\Phi(x) \right|\leq C\frac{\rho}{\sqrt{\ell'}} $$
where $\Phi(x)$ is the cumulative distribution function of the standard normal distribution and $C\leq 0.4748$ (see \cite{S}) is an absolute constant.
\end{thm}
Then we can state the following:
\begin{thm}\label{Z,ell,epsilon}
Let $A=\{v_1,v_2,\dots,v_n\}$ be distinct integers and let $Y$ be the sum of $\ell$ independent and uniformly distributed on $A$ variables $Y_1,\dots, Y_{\ell}$.
Then, there exists an absolute constant $D$ for which $$\left(\max_{x\in \mathbb{Z}} \mathbb{P}[Y=x]\right)\leq\frac{D}{n\sqrt{\ell-1}}.$$
In particular, however we take $\epsilon$, if $\ell$ is sufficiently large with respect to $\epsilon$ we have that
$$\left(\max_{x\in \mathbb{Z}} \mathbb{P}[Y=x]\right)\leq\frac{\epsilon}{n}.$$
\end{thm}
\proof
Due to Theorem \ref{LRthm}, we can assume $A=\{-(n-1)/2,\dots,(n-1)/2\}$.
We note that, due to the independence of the $Y_i$,
$$\mathbb{P}[Y=x]=$$
\begin{equation}\label{riduzione}\sum_{y\in [x-\frac{n-1}{2},x+\frac{n-1}{2}]}\mathbb{P}[Y_1+\dots+Y_{\ell-1}=y]\mathbb{P}[Y_{\ell}=x-y]=\frac{\mathbb{P}[Y_1+\dots+Y_{\ell-1}\in [x-\frac{n-1}{2},x+\frac{n-1}{2}]]}{n}.\end{equation}
Now we want to estimate $\mathbb{P}[Y_1+\dots+Y_{\ell-1}\in [x-\frac{n-1}{2},x+\frac{n-1}{2}]]$ through Theorem \ref{BE}.
Indeed, set
$$\Psi(z)=\mathbb{P}[\frac{1}{\sigma\sqrt{\ell-1}}\sum_{j=1}^{\ell-1} Y_j\leq z]$$
and set $(\Psi-\Phi)(z):=\Psi(z)-\Phi(z)$, we have:
$$\mathbb{P}[Y_1+\dots+Y_{\ell-1}\in [x-\frac{n-1}{2},x+\frac{n-1}{2}]]=\Psi\left(\frac{(x+\frac{n-1}{2})}{\sigma\sqrt{\ell-1}}\right)-\Psi\left(\frac{(x-\frac{n-3}{2})}{\sigma\sqrt{\ell-1}}\right)$$
\begin{equation}\label{Spezzata1}=\left(\Psi-\Phi\right)\left(\frac{(x+\frac{n-1}{2})}{\sigma\sqrt{\ell-1}}\right)-\left(\Psi-\Phi\right)\left(\frac{(x-\frac{n-3}{2})}{\sigma\sqrt{\ell-1}}\right)+\end{equation}
$$\Phi\left(\frac{(x+\frac{n-1}{2})}{\sigma\sqrt{\ell-1}}\right)-\Phi\left(\frac{(x-\frac{n-3}{2})}{\sigma\sqrt{\ell-1}}\right).$$
This implies that, due to the triangular inequality,
$$\mathbb{P}[Y_1+\dots+Y_{\ell-1}\in [x-\frac{n-1}{2},x+\frac{n-1}{2}]]\leq $$
\begin{equation}\label{Spezzata2}\left|\left(\Psi-\Phi\right)\left(\frac{(x+\frac{n-1}{2})}{\sigma\sqrt{\ell-1}}\right)\right|+\left|\left(\Psi-\Phi\right)\left(\frac{(x-\frac{n-3}{2})}{\sigma\sqrt{\ell-1}}\right)\right|+\end{equation}
$$\left|\Phi\left(\frac{(x+\frac{n-1}{2})}{\sigma\sqrt{\ell-1}}\right)-\Phi\left(\frac{(x-\frac{n-3}{2})}{\sigma\sqrt{\ell-1}}\right)\right|.$$
We note that the first two terms of Equation \eqref{Spezzata2} can be upper bounded via the Berry-Esseen theorem.
Here we have that
$$\sigma=\frac{n}{2\sqrt{3}}(1+o(1))$$
and
$$\rho=\frac{3\sqrt{3}}{4}(1+o(1)).$$
Therefore
\begin{equation}\label{primidue}\left|\left(\Psi-\Phi\right)\left(\frac{(x+\frac{n-1}{2})}{\sigma\sqrt{\ell-1}}\right)\right|+\left|\left(\Psi-\Phi\right)\left(\frac{(x-\frac{n-3}{2})}{\sigma\sqrt{\ell-1}}\right)\right|\leq 2C\frac{3\sqrt{3}}{4\sqrt{\ell-1}}(1+o(1)).\end{equation}
For the latter term, we note that
$$\left|\Phi\left(\frac{(x+\frac{n-1}{2})}{\sigma\sqrt{\ell-1}}\right)-\Phi\left(\frac{(x-\frac{n-3}{2})}{\sigma\sqrt{\ell-1}}\right)\right|\leq\max_x |\phi(x)| \frac{n}{\sigma\sqrt{\ell-1}}.$$
Hence, since the density function $\phi(\cdot)$ of the standard normal distribution has a maximum $\frac{1}{2\pi}$ in $0$, we have
\begin{equation}\label{terzotermine}
\left|\Phi\left(\frac{(x+\frac{n-1}{2})}{\sigma\sqrt{\ell-1}}\right)-\Phi\left(\frac{(x-\frac{n-3}{2})}{\sigma\sqrt{\ell-1}}\right)\right|\leq \frac{1}{2\pi}\frac{2\sqrt{3}}{\sqrt{l-1}}(1+o(1)).
\end{equation}
Summing up Equations \eqref{primidue} and \eqref{terzotermine}, we have that
$$\mathbb{P}[Y_1+\dots+Y_{\ell-1}\in [x-\frac{n-1}{2},x+\frac{n-1}{2}]]\leq \frac{D}{\sqrt{\ell-1}}$$
where $D$ is an absolute constant that also incorporates the term $o(1)$ which is bounded and goes to zero when $n$ goes to infinite.
Placing this bound in Equation \eqref{riduzione}, we obtain that
$$\mathbb{P}[Y=x]\leq \frac{D}{n\sqrt{\ell-1}}$$
which implies the thesis.
\endproof
Even though for $\ell=3$ the previous theorem provides a trivial bound, we can deal with this case directly. More precisely, we obtain the following non-trivial bound:
\begin{thm}\label{Z,3,epsilon}
Let $A=\{v_1,v_2,\dots,v_n\}$ be distinct integers, and let $Y_1, Y_2, Y_3$ be independent and uniformly distributed on $A$. Then, set $Y=Y_1+Y_2+Y_3$, we have that
$$\left(\max_{x\in \mathbb{Z}} \mathbb{P}[Y=x]\right)\leq\frac{3+1/n^2}{4n}.$$
\end{thm}
\proof
Due to Theorem \ref{LRthm}, we can assume $A=\{1,\dots,n\}$ (which is a translate of $\{-(n-1)/2,\dots,(n-1)/2\}$).
Also, Theorem \ref{LRthm} says that
$\mathbb{P}[Y=x]$ is maximal for $x=M=\lfloor\frac{3(n+1)}{2}\rfloor$ when $A=\{1,\dots,n\}$ (which corresponds to $x=0$ or $-\frac{1}{2}$ for $\{-(n-1)/2,\dots,(n-1)/2\}$).
Now we compute the number of triples
$$T=\{(y_1,y_2,y_3)\in A^3: y_1+y_2+y_3=M\}.$$

We consider explicitly only the case $n$ odd: the case $n$ even is completely analogous. Clearly $y_1+y_2+y_3=M$ implies that $y_1+y_2\in [M-n,M-1]=[\frac{n+3}{2},\frac{3n+1}{2}]$ and that $y_3=M-y_2-y_1$.

Given $x\in [\frac{n+3}{2},\frac{3n+1}{2}]$, we need to find the number of pair $(y_1,y_2)$ which sum to $x$. Here we have that, for any $y_2$ such that $1\leq x-y_2\leq n$, there exists exactly one $y_1$ such that $y_1+y_2=x$.
It follows that, if $x\leq n$ there are exactly $x-1$ pairs $y_1,y_2$ which sums to $x$. Similarly, if $x\geq n+1$, there are exactly $2n-(x-1)$ pairs $y_1,y_2$ which sums to $x$.
Noting that the number of pairs that sum to $x$ is equal to that of the pairs that sum to $2n+2-x$, we obtain that:
$$
|T|=\sum_{x\in [\frac{n+3}{2},n]}(x-1)+\sum_{x\in [n+1,\frac{3n+1}{2}]}(2n-(x-1))=
$$
\begin{equation}\label{numbertriple}2\sum_{x\in [\frac{n+3}{2},n]}(x-1)+n=2\left(\sum_{x=1}^{n-1} x-\sum_{x=1}^{\frac{n-1}{2}}x\right)+n=\frac{3n^2+1}{4}.\end{equation}
The thesis follows because
$$\mathbb{P}[Y=x]= \frac{\{(y_1,y_2,y_3)\in A^3: y_1+y_2+y_3=x\}}{n^3}\leq \frac{T}{n^3}\leq \frac{3+1/n^2}{4n}.$$
\endproof
\subsection{Known inequalities for $\mathbb{Z}_k$ and $\mathbb{Z}_p$}
In the previous discussion, we have presented an upper bound on $\mathbb{P}[Y=x]$ for $\mathbb{Z}$. It is then natural to investigate the case of $\mathbb{Z}_k$. Here, given a subset $A$ of size $n$, if the prime factors of $k$ are large enough, there exists a Freiman isomorphism of order $\ell$ from $A$ to a subset $B\subseteq \mathbb{Z}$. We recall that (see Tao and Vu, \cite{TV1}).
\begin{defn}
Let $k\geq 1$, and let $A$, $B$ be additive sets with ambient groups $V$ and $W$ respectively. A Freiman homomorphism of order
$\ell$, say $\phi$, from $(A, V)$ to $(B,W)$ (or more succinctly from $A$ to $B$) is a map $\phi : A \rightarrow B$ with the property that:
$$a_1+a_2+\dots+a_{\ell}=a_1'+a_2'+\dots+a_{\ell}'\rightarrow \phi(a_1)+\phi(a_2)+\dots+\phi(a_{\ell})=\phi(a_1')+\phi(a_2')+\dots+\phi(a_{\ell}').$$
If in addition there is an inverse map $\phi^{-1}: B \rightarrow A$ which is also a Freiman homomorphism of order $\ell$ from $(B,W)$ to $(A, V)$, then we say that $\phi$ is a Freiman isomorphism of order $\ell$ and that $(A, V)$ and $(B,W)$ are Freiman isomorphic of order $\ell$.
\end{defn}
Also from the result of Lev \cite{L3} (see also \cite{CDOR} and Theorem 3.1 of \cite{CP20} where similar bounds were given), we have that
\begin{thm}[Lev, \cite{L3}]
Let $A$ be a subset of size $n$ of $\mathbb{Z}_k$ where the prime factors of $k$ are larger than $n^n$. Then there exists a Freiman isomorphism of any order $\ell\leq n$, from $A$ to a subset $B\subseteq \mathbb{Z}$.
\end{thm}
Therefore, the result over the integers, namely Theorem \ref{Z,ell,epsilon}, implies that
\begin{thm}\label{Zk,ell,epsilon}
Let $A$ be a subset of size $n$ of $\mathbb{Z}_k$ where the prime factors of $k$ are larger than $n^n$. Then, set $Y=Y_1+\dots+Y_{\ell}$ the sum of $\ell$ independent uniformly distributed (over the same set $A$) random variables $Y_i$,
$$\mathbb{P}[Y=x]\leq \frac{D}{n\sqrt{\ell-1}}$$
where $D$ is an absolute constant.
\end{thm}
Note that this bound has quite a strong hypothesis on the prime factors of $k$. 
Another, quite general, inequality can be derived from the results of Lev (\cite{L1, L2}) who prove that
\begin{thm}[Lev, \cite{L1}]

The number of solutions of $a_1+a_2+\dots+a_{\ell}=x$, where $x\in \mathbb{Z}_p$, $a_i\in A\subseteq \mathbb{Z}_p$, and $|A|=n$ is at most
$$N_\lambda(A^{\ell})\leq \frac{1}{p}n^{\ell}+\sqrt{\frac{8}{\pi\ell}}n^{\ell-1}\left(1+2\ell^{-1/2}+(3/4)^{\ell/2+3}\ell^{3/2}\right).$$
\end{thm}
From this theorem, it immediately follows that
\begin{cor}\label{LevBound}
Let $A$ be a subset of size $n$ of $\mathbb{Z}_k$. Then, set $Y=Y_1+\dots+Y_{\ell}$ the sum of $\ell$ independent uniformly distributed (over the same set $A$) random variables $Y_i$,
$$\mathbb{P}[Y=x]\leq \frac{N_\lambda(A^{\ell})}{n^{\ell}}$$
and hence
$$\mathbb{P}[Y=x]\leq \frac{1}{p}+\frac{1}{n}\sqrt{\frac{8}{\pi\ell}}\left(1+2\ell^{-1/2}+(3/4)^{\ell/2+3}\ell^{3/2}\right).$$
\end{cor}
Now, set $\tilde{C}:=\sqrt{\frac{8}{\pi\ell}}\left(1+2\ell^{-1/2}+(3/4)^{\ell/2+3}\ell^{3/2}\right)$, the bound of Corollary \ref{LevBound} can be written as
$$\mathbb{P}[Y=x]\leq \frac{1}{p}+\frac{\tilde{C}}{n}$$
which has, for $p$ and $\ell$ large enough, the same behavior $\frac{\tilde{D}}{n\sqrt{\ell}}$ of the bound of Theorem \ref{Z,ell,epsilon}.
\section{Anticoncentration for small $\ell$}
In the previous Section, we have reviewed the known anticoncentration inequalities on $\mathbb{P}[Y=x]$. In particular, in $\mathbb{Z}_p$, the bound of Corollary \ref{LevBound} appears to be very powerful for large values of $\ell$. On the other hand, a direct computation shows that, if $\ell$ is small (i.e., $\ell<24$), the constant $\tilde{C}$ becomes larger than one, and that bound becomes trivial.
This section aims to provide a nontrivial bound also in this range. Now we consider a set $A=\{v_1,v_2,\dots,v_n\}$ of distinct elements of $\mathbb{Z}_p$, and the random variable $Y$ given by the sum of $\ell$ independent and uniformly distributed on $A$ variables $Y_1,\dots, Y_{\ell}$. The goal is here to provide an upperbound on $\left(\max_{x\in \mathbb{Z}_p} \mathbb{P}[Y=x]\right)$. First of all, we consider the case $\ell=3$, and we prove the following Lemma.
\begin{lem}\label{UniformSym}
Let $n\geq 2$, $p>2n$, $A=\{v_1,v_2,\dots,v_n\}$ be a set of distinct elements of $\mathbb{Z}_p$ such that $-v\in A$ whenever $v\in A$ and let $Y_1,Y_2$ and $Y_3$ be independent variables which are uniformly distributed on $A$. Then there exists an absolute constant $C_1<1$\footnote{Here the best approximation we have for this constant is $C_1< 0.99993.$}, such that, set $Y=Y_1+Y_2+Y_3,$ we have:
$$\left(\max_{x\in \mathbb{Z}_p} \mathbb{P}[Y=x]\right)\leq\frac{C_1}{n}.$$
\end{lem}
\proof
First of all, we note that, if $n=2$, we may suppose, without loss of generality, that $A=\{-1,1\}$. For this set, with the same proof of Theorem \ref{Z,3,epsilon}, we obtain that $\mathbb{P}[Y=x]\leq \frac{3+1/4}{4n}$ (which give a better constant than $\frac{C_1}{n}$). So, in the following, we only consider the case $n\geq 3$.

Following the proof of Proposition 6.1 of \cite{FKS}, set $f=\sum_{i=1}^n \frac{1}{n}\delta_{v_i}$, we have that its discrete Fourier transform is
$$ \hat{f}(k)=\sum_{i=1}^n \frac{1}{n}e^{-2\pi i v_ik/p}.$$
Then, the probability that $Y_1+Y_2=x$ can be written as
$$\mathbb{P}[Y_1+Y_2=x]=(f*f)(x)$$
$$=\frac{1}{p}\sum_{k=0}^{p-1} e^{2\pi ixk/p} (\hat{f}(k))^2$$
$$=\frac{1}{p}\sum_{k=0}^{p-1} e^{2\pi ixk/p} \left(\sum_{i=1}^n \frac{1}{n}e^{-2\pi i v_ik/p}\right)^2.$$
Here we note that, since the set $A$ is symmetric with respect to $0$, $$\mathbb{P}[Y_1+Y_2=x]=\mathbb{P}[Y_1+Y_2=-x].$$
Thus we can write
\begin{equation}\label{cos1}
\mathbb{P}[Y_1+Y_2=x]=\frac{1}{p}\sum_{k=0}^{p-1} \cos( 2\pi xk/p) \left(\sum_{i=1}^n \frac{1}{n}e^{-2\pi i v_ik/p}\right)^2.
\end{equation}
Since, due again to the symmetry of $A$, we also have that both $e^{-2\pi i v_ik/p}$ and $e^{2\pi i v_ik/p}$ appears in $\hat{f}(k)$, Equation \eqref{cos1} can be written as:
\begin{equation}\label{cos2}
\mathbb{P}[Y_1+Y_2=x]=\frac{1}{p}\sum_{k=0}^{p-1} \cos( 2\pi xk/p) \left(\sum_{i=1}^n \frac{1}{n}\cos(2\pi v_ik/p)\right)^2.
\end{equation}
Now we consider $B\subseteq \mathbb{Z}_p$ of cardinality $n$ such that
$\mathbb{P}[Y_1+Y_2\in B]$ is maximal and two positive real numbers $\epsilon_1,\epsilon_2\leq 1/24$. We divide $B$ into two sets defined as follows:
$$B_1:=\{x\in B: \mathbb{P}[Y_1+Y_2=x]\geq \frac{1-\epsilon_1}{n}\};$$
$$B_2:=B\setminus B_1=\{x\in B: \mathbb{P}[Y_1+Y_2=x]< \frac{1-\epsilon_1}{n}\}.$$
Now we divide the proof into two cases.

CASE 1: $|B_1|\geq (1-\epsilon_2)n$. Here we note that
\begin{equation}\label{val0}
\mathbb{P}[Y_1+Y_2=0]=\frac{1}{n}=\frac{1}{p}\sum_{k=0}^{p-1}(\hat{f}(k))^2.
\end{equation}
This means that, for the values $x \in B_1$, we can not lose too much by inserting the cosine. Indeed, for these values, we have that
\begin{equation}\label{valx}
\mathbb{P}[Y_1+Y_2=x]=\frac{1}{p}\sum_{k=0}^{p-1} \cos( 2\pi xk/p)(\hat{f}(k))^2\geq \frac{1-\epsilon_1}{n}.
\end{equation}
Taking the average over the values $x\in B_1$, we obtain
$$\frac{1}{|B_1|}\mathbb{P}[Y_1+Y_2\in B_1]=\frac{1}{|B_1|}\sum_{x\in B_1}\frac{1}{p}\sum_{k=0}^{p-1} \cos( 2\pi xk/p)(\hat{f}(k))^2=
$$
\begin{equation}\label{average}
=\frac{1}{p}\sum_{k=0}^{p-1} \left ((\hat{f}(k))^2 \frac{\sum_{x\in B_1}\cos( 2\pi xk/p)}{|B_1|}\right)=\frac{1}{p}+\frac{1}{p}\sum_{k=1}^{p-1} \left ((\hat{f}(k))^2 \frac{\sum_{x\in B_1}\cos( 2\pi xk/p)}{|B_1|}\right),
\end{equation}
where the last equalities hold by commuting the sums and noting that $\hat{f}(0)=1$.
Now, Equations \eqref{val0} and \eqref{valx} implies that
$$\frac{1}{|B_1|}\mathbb{P}[Y_1+Y_2\in B_1] \geq \frac{1-\epsilon_1}{n}=\frac{1}{p}\sum_{k=0}^{p-1}\left((\hat{f}(k))^2(1-\epsilon_1)\right)=$$ $$=\frac{1-\epsilon_1}{p}+\frac{1}{p}\sum_{k=1}^{p-1}\left((\hat{f}(k))^2(1-\epsilon_1)\right).$$
Now since $2n< p$, Equation \eqref{average} implies the existence of $\bar{k}\not=0$ such that \begin{equation}\label{bark}\sum_{x\in B_1}\frac{\cos( 2\pi x\bar{k}/p)}{|B_1|}\geq (1-2\epsilon_1).\end{equation}
Indeed, otherwise, we would have that
$$\frac{1}{p}+\frac{1}{p}\sum_{k=1}^{p-1} \left ((\hat{f}(k))^2 (1-2\epsilon_1)\right)> \frac{1}{p}+\frac{1}{p}\sum_{k=1}^{p-1} \left ((\hat{f}(k))^2 \frac{\sum_{x\in B_1}\cos( 2\pi xk/p)}{|B_1|}\right)\geq $$ $$ \frac{1-\epsilon_1}{p}+\frac{1}{p}\sum_{k=1}^{p-1}\left((\hat{f}(k))^2(1-\epsilon_1)\right)$$
that is, since $2n\leq p$,
$$\frac{\epsilon_1}{p}> \frac{1}{p}\sum_{k=1}^{p-1}(\hat{f}(k))^2\epsilon_1=\epsilon_1\left(\frac{1}{n}-\frac{1}{p} \right)\geq \frac{\epsilon_1}{p}$$
which is a contradiction. Also, note that since $\mathbb{Z}_p$ is a field, we may suppose, without loss of generality, that $\bar{k}=1$.

This means that
\begin{equation}\label{cos3}\sum_{x\in B_1}\frac{\cos( 2\pi x/p)}{|B_1|}\geq (1-2\epsilon_1).\end{equation}
Hence the set
$$B_1':=\{x\in B_1: x/p\in [-1/6,1/6]\}=\{x\in B_1: \cos( 2\pi x/p)\in [1/2,1]\}$$ has size larger than $(1-2\epsilon_2-4\epsilon_1)n$.
Indeed, since $|B|\leq n$, $|B_1'|<(1-2\epsilon_2-4\epsilon_1)n$ would imply that
$$\frac{n+(1-2\epsilon_2-4\epsilon_1)n}{2} \geq \frac{|B_1|+|B_1'|}{2}=|B_1'|+\frac{|B_1|-|B_1'|}{2}.$$
Then we would obtain
$$(1-\epsilon_2-2\epsilon_1)n>|B_1'|+\frac{|B_1\setminus B_1'|}{2}\geq \sum_{x\in B_1}\cos( 2\pi x/p)\geq (1-2\epsilon_1)(1-\epsilon_2)n$$
which is a contradiction.

Now we come back to the estimation of $\max_{x\in \mathbb{Z}_p}\mathbb{P}[Y=x]$. Because of the independence of the variables $Y_1,Y_2$ and $Y_3$, we have
\begin{equation}\label{independence}
\mathbb{P}[Y=x]=\sum_{v_i\in A}\mathbb{P}[Y_1+Y_2=x-v_i]\mathbb{P}[Y_3=v_i].
\end{equation}
Now we split this computation according to whether $x-v_i$ belongs to $B_1'$ to $B_1'':=B_1\setminus B_1'$ or $B_2$.
\begin{equation}\label{split1}
\sum_{v_i\in A:\ x-v_i\in B_1'}\mathbb{P}[Y_1+Y_2=x-v_i]\mathbb{P}[Y_3=v_i]+\sum_{v_i\in A:\ x-v_i\in B_1''}\mathbb{P}[Y_1+Y_2=x-v_i]\mathbb{P}[Y_3=v_i]+
\end{equation}
$$+\sum_{v_i\in A:\ x-v_i\in B_2}\mathbb{P}[Y_1+Y_2=x-v_i]\mathbb{P}[Y_3=v_i].$$
Due to Equations \eqref{val0} and \eqref{valx}, $\mathbb{P}[Y_1+Y_2=x-v_i]$ is always bounded by $1/n$. Then, since $|\{v_i\in A:\ x-v_i\in B_1''\}|\leq (4\epsilon_1+2\epsilon_2)n$, we can provide the following bounds:
$$\sum_{v_i\in A:\ x-v_i\in B_1'}\mathbb{P}[Y_1+Y_2=x-v_i]\mathbb{P}[Y_3=v_i]\leq\frac{1}{n} \frac{|\{v_i\in A:\ x-v_i\in B_1'\}|}{n};$$
$$\sum_{v_i\in A:\ x-v_i\in B_1''}\mathbb{P}[Y_1+Y_2=x-v_i]\mathbb{P}[Y_3=v_i]\leq \frac{4\epsilon_1+2\epsilon_2}{n};$$
and
$$\sum_{v_i\in A:\ x-v_i\in B_2}\mathbb{P}[Y_1+Y_2=x-v_i]\mathbb{P}[Y_3=v_i]<\frac{1-\epsilon_1}{n}\frac{|\{v_i\in A:\ x-v_i\in B_2\}|}{n}\leq \frac{\epsilon_2(1-\epsilon_1)}{n}.$$
Therefore $\mathbb{P}[Y=\bar{x}]\geq\frac{1-\epsilon_2}{n}+\frac{\epsilon_2(1-\epsilon_1)}{n}$ (the right-hand side of this inequality will be denoted in the following by $\frac{C_1}{n}$) implies that
$$\frac{1}{n} \frac{|\{v_i\in A:\ \bar{x}-v_i\in B_1'\}|}{n}+\frac{4\epsilon_1+2\epsilon_2}{n}+\frac{\epsilon_2(1-\epsilon_1)}{n^2}\geq \frac{1-\epsilon_2}{n}+\frac{\epsilon_2(1-\epsilon_1)}{n}.$$
As a consequence, we have
$$|\{v_i\in A:\ \bar{x}-v_i\in B_1'\}|\geq (1-4\epsilon_1-3\epsilon_2)n$$
and hence, set
$$A_1:=\{v_i\in A:\ v_i\in [-1/6-\bar{x},1/6-\bar{x}]\},$$
we have that
\begin{equation}\label{vi}
|A_1|\geq |\{v_i\in A:\ \bar{x}-v_i\in B_1'\}|\geq (1-4\epsilon_1-3\epsilon_2)n.
\end{equation}
On the other hand, given two triples $v_1,v_2,v_3$ and $v_1',v_2',v_3'\in A_1$, we have that $v_1+v_2+v_3=v_1'+v_2'+v_3'$ if and only if they have the same sum also in $\mathbb{Z}$ (with the trivial identification).
We note that, since $\epsilon_1,\epsilon_2\leq 1/24$ and $n\geq 3$, $n (1-4\epsilon_1-3\epsilon_2)>2$.
Hence, due to Theorem \ref{Z,3,epsilon},
$$\max_{x\in \mathbb{Z}_p} \mathbb{P}[Y=x|Y_1,Y_2,Y_3\in A_1]\leq \frac{3+1/9}{4n (1-4\epsilon_1-3\epsilon_2)}.$$
Here we have that
$$\mathbb{P}[Y=x|Y_1\not\in A_1]=\mathbb{P}[Y=x|Y_2\not\in A_1]=\mathbb{P}[Y=x|Y_3\not\in A_1],$$
and that $\mathbb{P}[Y=x|Y_1\not\in A_1]$ can be written as
$$\sum_{y\in \mathbb{Z}_p}\mathbb{P}[Y_1+Y_2=y|Y_1\not\in A_1]\mathbb{P}[Y_3=x-y]\leq \sum_{y\in \mathbb{Z}_p}\mathbb{P}[Y_1+Y_2=y|Y_1\not\in A_1] \frac{1}{n}=\frac{1}{n}.$$
Hence we can bound $\mathbb{P}[Y=x]$ as follows:
$$\mathbb{P}[Y=x]\leq \mathbb{P}[Y=x|Y_1,Y_2,Y_3\in A_1]+3\mathbb{P}[Y=x|Y_1\not\in A_1]\mathbb{P}[Y_1\not\in A_1]\leq$$
\begin{equation}\label{case1}
\frac{3+1/9}{4n (1-4\epsilon_1-3\epsilon_2)}+\frac{3}{n}\frac{|A\setminus A_1|}{n}\leq \frac{3+1/9}{4n (1-4\epsilon_1-3\epsilon_2)} +3\frac{4\epsilon_1+3\epsilon_2}{n}.
\end{equation}
Now we consider the pair $(\epsilon_1,\epsilon_2)$ to satisfy
$$\frac{3+1/9}{4n (1-4\epsilon_1-3\epsilon_2)} +3\frac{4\epsilon_1+3\epsilon_2}{n}\leq \frac{1-\epsilon_2}{n}+\frac{\epsilon_2(1-\epsilon_1)}{n}.$$ Note that the set of these pairs is clearly non-empty, since $(0,0)$ satisfies this relation with the strong inequality. Hence, there is a nontrivial set of positive pairs $\epsilon_1,\epsilon_2\leq 1/24$ which satisfy it.
For these pairs, we obtain that
$$\mathbb{P}[Y=x]\leq \frac{1-\epsilon_2}{n}+\frac{\epsilon_2(1-\epsilon_1)}{n}=\frac{C_1}{n}$$
and we conclude CASE 1 by computing (with Mathematica) the minimum possible $C_1$ in this range, which is $C_1< 0.99993$.

CASE 2: $|B_1|< (1-\epsilon_2)n$. Here we note that Equation \eqref{independence} can be split as follows:
$$\mathbb{P}[Y=x]=$$
\begin{equation}\label{split2}
\sum_{v_i\in A:\ x-v_i\in B_1}\mathbb{P}[Y_1+Y_2=x-v_i]\mathbb{P}[Y_3=v_i]+\sum_{v_i\in A:\ x-v_i\in B_2}\mathbb{P}[Y_1+Y_2=x-v_i]\mathbb{P}[Y_3=v_i].
\end{equation}
Now we can provide the following bounds:
$$\sum_{v_i\in A:\ x-v_i\in B_1}\mathbb{P}[Y_1+Y_2=x-v_i]\mathbb{P}[Y_3=v_i]\leq \frac{|\{v_i\in A:\ x-v_i\in B_1\}|}{n}\frac{1}{n}$$
and
$$\sum_{v_i\in A:\ x-v_i\in B_2}\mathbb{P}[Y_1+Y_2=x-v_i]\mathbb{P}[Y_3=v_i]\leq \frac{|\{v_i\in A:\ x-v_i\in B_2\}|(1-\epsilon_1)}{n}\frac{1}{n}.$$
It follows that
\begin{equation}\label{split3}
\mathbb{P}[Y=x]\leq \frac{|\{v_i\in A:\ x-v_i\in B_1\}|}{n}\frac{1}{n}+ \frac{|\{v_i\in A:\ x-v_i\in B_2\}|(1-\epsilon_1)}{n}\frac{1}{n}.
\end{equation}
We note that the right-hand side of Equation \eqref{split3} increases with $|B_|$. 
Therefore, since $|B_1|< (1-\epsilon_2)n$, we upper-bound the right hand side of Equation \eqref{split3} by assuming that $|\{v_i\in A:\ x-v_i\in B_1\}|= (1-\epsilon_2)n$ and
$|\{v_i\in A:\ x-v_i\in B_2\}|=\epsilon_2 n.$
Summing up, we obtain that
\begin{equation}\label{case2}
\mathbb{P}[Y=x]\leq \frac{1-\epsilon_2}{n}+\frac{\epsilon_2(1-\epsilon_1)}{n}=\frac{C_1}{n}
\end{equation}
which concludes CASE 2.
\endproof

\begin{prop}\label{Uniform}
Let $n\geq 2$, $p>2n$, $A=\{v_1,v_2,\dots,v_n\}$ be distinct elements of $\mathbb{Z}_p$ and let $Y_1,Y_2$ and $Y_3$ be independent variables which are uniformly distributed on $A$. Then there exists $C_2<1$\footnote{Here the best approximation we have for this constant is $C_2< 0.999986.$}, such that, set $Y=Y_1+Y_2+Y_3$, we have:
$$\left(\max_{x\in \mathbb{Z}_p} \mathbb{P}[Y=x]\right)\leq\frac{C_2}{n}.$$
\end{prop}
\proof
First of all, we note that, if $n=2$, we may suppose, without loss of generality, that $A=\{-1,1\}$. For this set, with the same proof of Theorem \ref{Z,3,epsilon}, we obtain that $\mathbb{P}[Y=x]\leq \frac{3+1/4}{4n}$ (which give a better constant than $\frac{C_2}{n}$). So, in the following, we only consider the case $n\geq 3$.

Let $\bar{x}$ be such that $\mathbb{P}[Y_1+Y_2=\bar{x}]$ is maximal. Since $\mathbb{Z}_p$ is a field, we may suppose, without loss of generality, that $\bar{x}=0$.
Set $\epsilon_3$ such that
$$\frac{C_1}{(1-\epsilon_{3})}+3\epsilon_{3}=1-\epsilon_{3}$$ and $C_2=1-\epsilon_3$. We note that this relation has a solution $\epsilon_3\in (0,1-C_1)$: indeed, in $0$ its left-hand side is smaller than its right-hand side, while in $1-C_1$ the opposite inequality holds. Also, we can compute, computationally, the values of $\epsilon_3$ and $C_2$, and note that $C_2<0.999986$. Then we divide the proof into two cases.

CASE 1: $\mathbb{P}[Y_1+Y_2=0]\geq \frac{1-\epsilon_{3}}{n}$.
Here we define
$$A_1:=\{x\in A:\ -x\in A\}.$$
We have that
$$\mathbb{P}[Y_1+Y_2=0]=\mathbb{P}[Y_1\in A_1]\cdot \mathbb{P}[Y_2=-Y_1]=\frac{|A_1|}{n^2}.$$
This implies that
\begin{equation}\label{sizeA1}|A_1|\geq n(1-\epsilon_{3}).\end{equation}
Note that, since $n\geq 3$, $|A_1|\geq 2$. Hence, due to Lemma \ref{UniformSym}, there exists $C_1$ such that
$$\mathbb{P}[Y=x|Y_1,Y_2,Y_3\in A_1]\leq \frac{C_1}{n(1-\epsilon_{3})}.$$
Here we have that
$$\mathbb{P}[Y=x|Y_1\not\in A_1]=\mathbb{P}[Y=x|Y_2\not\in A_1]=\mathbb{P}[Y=x|Y_3\not\in A_1],$$
and that $\mathbb{P}[Y=x|Y_1\not\in A_1]$ can be written as
$$\sum_{y\in \mathbb{Z}_p}\mathbb{P}[Y_1+Y_2=y|Y_1\not\in A_1]\mathbb{P}[Y_3=x-y]\leq \sum_{y\in \mathbb{Z}_p}\mathbb{P}[Y_1+Y_2=y|Y_1\not\in A_1] \frac{1}{n}=\frac{1}{n}.$$
Therefore, we can bound $\mathbb{P}[Y=x]$ as follows:
$$\mathbb{P}[Y=x]\leq \mathbb{P}[Y=x|Y_1,Y_2,Y_3\in A_1]+3\mathbb{P}[Y=x|Y_1\not\in A_1]\mathbb{P}[Y_1\not\in A_1]\leq$$
\begin{equation}\label{case1b}
\frac{C_1}{n(1-\epsilon_{3})}+\frac{3}{n}\frac{|A\setminus A_1|}{n}\leq \frac{C_1}{n(1-\epsilon_{3})}+3\frac{\epsilon_{3}}{n}=\frac{C_2}{n}
\end{equation}
which concludes CASE 1.

CASE 2: $\mathbb{P}[Y_1+Y_2=0]< \frac{1-\epsilon_{3}}{n}$. Here we have that
\begin{equation}\label{split3}
\mathbb{P}[Y=x]=\sum_{v_i\in A}\mathbb{P}[Y_1+Y_2=x-v_i]\mathbb{P}[Y_3=v_i].
\end{equation}
Since
$$\mathbb{P}[Y_1+Y_2=x-v_i]\leq \mathbb{P}[Y_1+Y_2=0]<\frac{1-\epsilon_{3}}{n},$$ Equation \eqref{split3} can be written as
\begin{equation}\label{case2b}
\mathbb{P}[Y=x]<\sum_{v_i\in A}\frac{1-\epsilon_{3}}{n}\mathbb{P}[Y_3=v_i]=\frac{1-\epsilon_{3}}{n}=\frac{C_2}{n}.
\end{equation}
\endproof
This proposition can be generalized to:
\begin{thm}\label{General3}
Let $\lambda\leq \frac{9}{10}$, $p> \frac{2}{\lambda}$, $Y_1,Y_2$ and $Y_3$ be identical and independent variables such that $$\left(\max_{x\in \mathbb{Z}_p} \mathbb{P}[Y_1=x]\right)\leq \lambda.$$ Then there exists $C_3<1$\footnote{Here the best approximation we have for this constant is $C_3< 1-1.3\cdot 10^{-12}.$}, such that, set $Y=Y_1+Y_2+Y_3$ we have:
$$\left(\max_{x\in \mathbb{Z}_p} \mathbb{P}[Y=x]\right)\leq C_3\lambda.$$
\end{thm}
\proof
We consider positive $\epsilon_{4}$ and $\epsilon_5$ which satisfy the relation:
$$1-\epsilon_4\epsilon_5\geq 3(\epsilon_4+\epsilon_5)+\frac{C_2}{(1-\epsilon_4)^3(1-\epsilon_5)}$$
and we set $C_3:=1-\epsilon_4\epsilon_5$. We will also assume that $\epsilon_5<1/10$ and $\epsilon_4<1$. It is clear that, since $(0,0)$ satisfies the above relation (with a strong inequality), the set of pairs subject to these constraints is nonempty.

Then, we define
$$A_1:=\{x\in \mathbb{Z}_p: \mathbb{P}[Y_1=x]\geq \lambda(1-\epsilon_{4})\}.$$
Then we divide the proof into two cases according to the cardinality of $A_1$.

CASE 1: $|A_1|\geq \frac{(1-\epsilon_{5})}{\lambda}.$
We want to provide an upper bound to $\mathbb{P}[Y=x|Y_1,Y_2,Y_3\in A_1]$ by applying Proposition \ref{Uniform}. For this purpose, we note that:
$$\mathbb{P}[Y=x|Y_1,Y_2,Y_3\in A_1]=$$
\begin{equation}\label{Stima1}
\sum_{y_1,y_2\in A_1}\mathbb{P}[Y_1=y_1|Y_1\in A_1]\mathbb{P}[Y_2=y_2|Y_2\in A_1]\mathbb{P}[Y_3=x-y_1-y_2|Y_3\in A_1].
\end{equation}
Here, for $y_1\in A_1$, we have
$$
\mathbb{P}[Y_1=y_1|Y_1\in A_1]\mathbb{P}[Y_1\in A_1]=\mathbb{P}[Y_1=y_1]\leq \lambda.
$$
Since
$$\mathbb{P}[Y_1\in A_1]=\sum_{x\in A_1} \mathbb{P}[Y_1=x]\geq |A_1| \lambda (1-\epsilon_{4}),$$
it follows that
$$
\mathbb{P}[Y_1=y_1|Y_1\in A_1]\leq \frac{\lambda}{\mathbb{P}[Y_1\in A_1]}\leq \frac{1}{|A_1|(1-\epsilon_{4})}.
$$
Note that, named by $\tilde{Y}=\tilde{Y_1}+\tilde{Y_2}+\tilde{Y_3}$ where $\tilde{Y_1}, \tilde{Y_2}, \tilde{Y_3}$ are uniform distribution over $A_1$, we have that
\begin{equation}\label{Bound3}
\mathbb{P}[Y_1=y_1|Y_1\in A_1] \leq \frac{1}{|A_1|(1-\epsilon_4)}=\frac{1}{(1-\epsilon_4)}\mathbb{P}[\tilde{Y_1}=y_1].\end{equation}
Noting that $\mathbb{P}[\tilde{Y_1}=y_1]\leq \frac{\lambda}{(1-\epsilon_5)}$, Equations \eqref{Stima1} and \eqref{Bound3} imply that
\begin{equation}\label{Bound4}
\mathbb{P}[Y=x|Y_1,Y_2,Y_3\in A_1]\leq \left(\frac{1}{1-\epsilon_4}\right)^3\mathbb{P}[\tilde{Y}=x]\leq \left(\frac{1}{1-\epsilon_4}\right)^3 \frac{C_2\lambda}{(1-\epsilon_{5})}.
\end{equation}
Here the last inequality holds because, since $\lambda\leq 9/10$, $\epsilon_5<1/10$, and we are in CASE 1, we have $|A_1|>1$. Moreover, since $|A_1|$ is integer, this means that $|A_1|\geq 2$ and we can apply Proposition \ref{Uniform} to the distribution $\tilde{Y}$.

Recalling that $$\mathbb{P}[Y_1\in A_1]\geq |A_1| \lambda (1-\epsilon_{4})\geq (1-\epsilon_5)(1-\epsilon_4)>(1-\epsilon_4-\epsilon_5),$$
we have $\mathbb{P}[Y_1\not\in A_1]\leq \epsilon_{4}+\epsilon_5$.
Proceeding as in Equation \eqref{case1b}, it follows that $$\mathbb{P}[Y=x]\leq \mathbb{P}[Y=x|Y_1,Y_2,Y_3\in A_1]+3\mathbb{P}[Y=x|Y_1\not\in A_1]\mathbb{P}[Y_1\not\in A_1]\leq$$
\begin{equation}\label{case1c}
\left(\frac{1}{1-\epsilon_4}\right)^3 \frac{C_2\lambda}{(1-\epsilon_{5})}+3\lambda\mathbb{P}[Y_1\not\in A_1]\leq \left(\frac{1}{1-\epsilon_4}\right)^3 \frac{C_2\lambda}{(1-\epsilon_{5})}+3\lambda(\epsilon_{4}+\epsilon_5)\leq C_3 \lambda
\end{equation}
which concludes CASE 1.

CASE 2: $|A_1|< \frac{(1-\epsilon_{5})}{\lambda}.$
Here we prove that $\mathbb{P}[Y_1+Y_2=x]\leq C_3 \lambda.$
Indeed, we have that
$$\mathbb{P}[Y_1+Y_2=x]=\sum_{y_1\in A_1}\mathbb{P}[Y_1=y_1]\mathbb{P}[Y_2=x-y_1]+\sum_{y_1\not\in A_1}\mathbb{P}[Y_1=y_1]\mathbb{P}[Y_2=x-y_1].$$
Here we note that
\begin{equation}
\mathbb{P}[Y_1=y_1]\leq \begin{cases}
\lambda \mbox{ if } y_1\in A_1;\\
\lambda(1-\epsilon_{4}) \mbox{ if } y_1\not\in A_1.
\end{cases}
\end{equation}
Therefore
$$\mathbb{P}[Y_1+Y_2=x]\leq \lambda \left(\sum_{y_1\in A_1}\mathbb{P}[Y_2=x-y_1]+ (1-\epsilon_{4})\sum_{y_1\not\in A_1}\mathbb{P}[Y_2=x-y_1]\right)=$$
\begin{equation}\label{case2c}
=\lambda - \lambda\epsilon_{4}\sum_{y_1\not\in A_1}\mathbb{P}[Y_2=x-y_1].
\end{equation}
Now we note that, named $A_2:=\{y_1\in \mathbb{Z}_p:\ y_1\not \in A_1\}$, $|A_2|\geq p-\frac{(1-\epsilon_{5})}{\lambda}$ and we have that
$$\mathbb{P}[Y_2\in A_2]=1- \mathbb{P}[Y_2\in A_1]\geq 1- \lambda|A_1|\geq \epsilon_5.$$
Also, given another set $B$ of the same cardinality, we have that
$$\mathbb{P}[Y_2\in B]\geq \mathbb{P}[Y_2\in A_2]\geq \epsilon_{5}.$$
Therefore, considering that $|\{z\in \mathbb{Z}_p:\ z=x-y_1,\ y_1\not\in A_1\}|=|A_2|$, Equation \eqref{case2c} can be written as
$$\mathbb{P}[Y_1+Y_2=x]\leq\lambda-\lambda\epsilon_4\mathbb{P}[Y_2\in\{z\in \mathbb{Z}_p:\ z=x-y_1,\ y_1\not\in A_1\}]\leq \lambda - \lambda\epsilon_{4}\epsilon_{5}=C_3\lambda.$$
Now it is enough to note that, since $Y_1,Y_2$ and $Y_3$ are independent,
$$\mathbb{P}[Y_1+Y_2+Y_3=x]=\sum_{y_3\in \mathbb{Z}_p}\mathbb{P}[Y_3=y_3]\mathbb{P}[Y_1+Y_2=x-y_3]\leq \sum_{y_3\in \mathbb{Z}_p}\mathbb{P}[Y_3=y_3]C_3\lambda\leq C_3\lambda$$
which concludes CASE 2.

The thesis follows showing, with Mathematica, that we can choose $C_3< 1-2.27\cdot 10^{-12}.$
\endproof
\begin{rem}
In Theorem \ref{General3}, we may also weaken the hypothesis that $\lambda<9/10$ and assume that $\lambda<1$. With the same proof, choosing $\epsilon_5<1-\lambda$, we find the existence of a constant $C_{3,\lambda}<1$ also for values of $\lambda$ that are close to one. On the other hand, this constant cannot be made explicit, and it depends on $\lambda$. So, to obtain an absolute constant, we prefer to assume $\lambda$ smaller than a given value (i.e. $9/10$).
\end{rem}
As a consequence, we can state the following result, which is analogous, for the case of $\mathbb{Z}_p$, to that of the previous section.
\begin{thm}\label{Generalell}
Let us consider $\lambda\leq \frac{9}{10}$, $\epsilon>0$, $p>\frac{2}{\lambda\epsilon}$ and let $Y=Y_1+\dots+Y_{\ell}$ where $Y_i$ are i.i.d. whose distributions are upper-bounded by $\lambda$.
Then, if $\ell$ is sufficiently large with respect to $\epsilon$, we have that
\begin{equation}\label{LastEq}\left(\max_{x\in \mathbb{Z}_p} \mathbb{P}[Y=x]\right)\leq \epsilon \lambda.\end{equation}
In particular, if $k_0$ is a positive integer such that $(C_3)^{k_0}<\epsilon\leq (C_3)^{k_0-1}$ and $\ell_0=3^{k_0}$, Equation \eqref{LastEq} holds for any $\ell\geq \ell_0$.
\end{thm}
\proof
First of all, we prove, by induction on $k$, that, assuming $p>\frac{2}{C_3^{k-1}\lambda}$,
\begin{equation}\label{induction}\left(\max_{x\in \mathbb{Z}} \mathbb{P}[Y_1+Y_2+\dots+Y_{3^k}=x]\right)\leq C_3^k\lambda.\end{equation}
BASE CASE. The case $k=1$ follows from Theorem \ref{General3}.

INDUCTIVE STEP. We assume $p>\frac{2}{C_3^{k}\lambda}$ and that Equation \eqref{induction} is true for $k$ and we prove it for $k+1$.
At this purpose we set $\tilde{Y}_1=Y_1+Y_2+\dots+Y_{3^k}$ and we note that, set $\tilde{\lambda}=C_3^k\lambda$ for the inductive hypothesis,
$$\mathbb{P}[\tilde{Y}_1=x]\leq \tilde{\lambda}.$$
Therefore, because of Theorem \ref{General3}, we see that if $\tilde{Y}_2$, $\tilde{Y}_3$, are three identical copies of $\tilde{Y_1}$ and the set $Y= \tilde{Y}_1 +\tilde{Y}_2+\tilde{Y}_3$,
$$\mathbb{P}[Y=x]\leq C_3\tilde{\lambda}=C_3^{k+1}\lambda.$$
The inductive claim follows since $Y= \tilde{Y}_1 +\tilde{Y}_2+\tilde{Y}_3=Y_1+Y_2+\dots+Y_{3^{k+1}}$.

Now we consider $\ell \geq 3^k$. A bound for this case follows, since we have
$$\mathbb{P}[Y=x]=\sum_{y_3}\mathbb{P}[Y_{3^k+1}+Y_{3^k+2}+\dots+Y_{\ell}=y]\mathbb{P}[Y_1+Y_2+\dots +Y_{3^k}=x-y]\leq$$
$$ \sum_{y}\mathbb{P}[Y_{3^k+1}+Y_{3^k+2}+\dots+Y_{\ell}=y] C_3^k\lambda\leq C_3^k\lambda.$$

Now, define $k_0$ to be a positive integer such that $(C_3)^{k_0}<\epsilon\leq (C_3)^{k_0-1}$ and let $\ell_0=3^{k_0}$. Since 
$$p>\frac{2}{\epsilon\lambda}>\frac{2}{\lambda C_3^{k_0-1}},$$
Equation \eqref{LastEq} holds for any $\ell\geq \ell_0$ and the main claim of the theorem follows.
\endproof
\begin{rem}
In Theorem \ref{Generalell}, we may also weaken the hypothesis that $\lambda<9/10$ and assume that $\lambda<1$. With the same proof, we find the existence of a constant $C_{3,\lambda}<1$ for which the following statement holds. 

Let us consider $\epsilon>0$, $p>\frac{2}{\lambda\epsilon}$ and let $Y=Y_1+\dots+Y_{\ell}$ where $Y_i$ are i.i.d. whose distributions are upper-bounded by $\lambda$.
Then, if $\ell$ is sufficiently large with respect to $\epsilon$, we have that
\begin{equation}\label{LastEq2}\left(\max_{x\in \mathbb{Z}_p} \mathbb{P}[Y=x]\right)\leq \epsilon \lambda.\end{equation}
In particular, if $k_0$ is a positive integer such that $(C_{3,\lambda})^{k_0}<\epsilon\leq (C_{3,\lambda})^{k_0-1}$ and $\ell_0=3^{k_0}$, Equation \eqref{LastEq2} holds for any $\ell\geq \ell_0$.
\end{rem}

\subsection{Conclusive Remarks}
The importance of this result is that it is nontrivial also for small values of $\ell\geq 3$ provided that $\epsilon$ is sufficiently close to, but still below, $1$. We recall that, on the other hand, the bound of Corollary \ref{LevBound} is worse that $1/n$ whenever $\ell<24$.

With our bound (assuming $p$ and $\lambda$ satisfy the hypotheses of Theorem \ref{Generalell}), if $3\leq\ell<9$ we obtain that
$$\left(\max_{x\in \mathbb{Z}_p} \mathbb{P}[Y=x]\right)\leq C_3 \lambda$$
while, if $9\leq \ell <27,$
$$\left(\max_{x\in \mathbb{Z}_p} \mathbb{P}[Y=x]\right)\leq (C_3)^2 \lambda.$$

Moreover, set $\nu=-\log_3{C_3}$, Theorem \ref{Generalell} can be restated as follows.

Let us consider $\lambda\leq \frac{9}{10}$, $p> \frac{2}{\lambda}\left(\frac{\ell_0}{3}\right)^{\nu}$, where $\ell_0$ is a power of three and $\ell\geq \ell_0$, and let $Y=Y_1+\dots+Y_{\ell}$ where $Y_i$ are i.i.d. whose distributions are upper-bounded by $\lambda$.
Then, the distribution $Y$ is bounded by $\lambda\left(\frac{3}{\ell_0}\right)^{\nu}$, where $\nu$ is a positive absolute constant.

From this discussion, we can also estimate that $\nu>1.19\cdot 10^{-12}$. The value obtained here (that we do not expect to be tight) is a result of the universal nature of our estimations which must hold for any distribution bounded by $9/10$. While $\nu$ is a positive absolute constant (which makes this result nontrivial), we expect that it would be possible to improve this value for specific classes of distributions.
\section*{Acknowledgements} The author would like to thank Stefano Della Fiore for our useful discussions on this
topic. The author was partially supported by INdAM--GNSAGA.

\end{document}